\documentclass[11pt]{article}

\usepackage{fullpage}
\usepackage[utf8]{inputenc}
\usepackage[T1]{fontenc}
\usepackage{amsmath,amsfonts,amssymb,amsthm,bbm}
\usepackage{stackrel}
\usepackage{enumerate}
\usepackage{graphicx}
\usepackage{hyperref}
\usepackage{subfigure}
\usepackage{color}
\usepackage{stmaryrd}
\usepackage{amsmath}
\usepackage{pifont}
\usepackage{stmaryrd}
\usepackage[dvipsnames]{xcolor}

\newtheorem{theorem}{Theorem}
\newtheorem{corollary}[theorem]{Corollary}
\newtheorem{proposition}[theorem]{Proposition}

\newtheorem{lemma}[theorem]{Lemma}
\newtheorem{definition}[theorem]{Definition}

\newtheorem{remark}[theorem]{Remark}

\numberwithin{theorem}{section}
\numberwithin{equation}{section}
\numberwithin{figure}{section}

\newcommand{\eps}{\varepsilon}
\newcommand{\ind}{\mathbbm{1}}

\newcommand{\wt}{\widetilde}

\newcommand{\mphi}{\widetilde\varphi}
\newcommand{\mB}{\widetilde B}
\newcommand{\mZ}{\widetilde{\mathbb Z}^2}

\newcommand{\R}{\mathbbm{R}}
\newcommand{\Z}{\mathbbm{Z}}

\newcommand{\D}{\mathbbm{D}}

\newcommand{\Hb}{\mathbbm{H}}
\newcommand{\Pb}{\mathbbm{P}}

\newcommand{\Ac}{\mathcal{A}}
\newcommand{\Bc}{\mathcal{B}}
\newcommand{\cC}{\mathcal{C}}
\newcommand{\Cc}{\mathcal{C}}

\newcommand{\Lc}{\mathcal{L}}

\newcommand{\Nc}{\mathcal{N}}
\newcommand{\Ec}{\mathcal{E}}

\newcommand{\loc}{\mathrm{loc}}
\newcommand{\dist}{\mathrm{dist}}
\newcommand{\out}{\mathrm{out}}

\newcommand{\connect}[1]{\stackrel{#1}\longleftrightarrow}

\title{Arm events in critical planar loop soups}
\author{Yijie Bi\footnote{\url{2200010925@stu.pku.edu.cn}; Peking University} \and  Yifan Gao\footnote{\url{yifangao@cityu.edu.hk}; City University of Hong Kong} \and Pierre Nolin\footnote{\url{bpmnolin@cityu.edu.hk}; City University of Hong Kong} \and Wei Qian\footnote{\url{weiqian0@hku.hk}; University of Hong Kong}}
\date{}

\begin{document}

\maketitle

\begin{abstract}
We establish up-to-constants estimates for arm events in the Brownian loop soup on the 2D metric graph associated with the square lattice. More specifically, we consider two natural geometric events: first, ``bulk'' four-arm events, corresponding to two large connected components of loops getting close to each other; and then, two-arm events in the half-plane, used to estimate the probability that a cluster of loops approaches the boundary. Our proof relies on an estimate by Lupu-Werner \cite{lupu2018random}, thanks to the well-known coupling between the loop soup and the Gaussian free field on the metric graph \cite{MR2815763,lupu2016loop}.

As a consequence, we also obtain up-to-constant upper bounds for the corresponding arm events in the random walk loop soup on the square lattice. In this way, we verify Assumptions~5.7 and~5.11 in \cite{GNQ2}: in a box with side length $N$, this implies the existence of crossings where the Gaussian free field remains below $a \sqrt{\log \log N}$ in absolute value, for some constant $a > 0$ large enough.
\end{abstract}

\section{Introduction}

In this paper, we study the Brownian loop soup (BLS) on a planar metric graph, which also yields consequences for the random walk loop soup (RWLS) on the corresponding discrete graph. In particular, we prove two assumptions stated in \cite{GNQ2}: in that paper, these assumptions are shown to have implications on the percolation of two-sided level sets in the discrete Gaussian free field.
We restrict ourselves to the loop soup with critical intensity $1/2$, which is intimately related to the Gaussian free field (GFF), thanks to Le Jan's isomorphism \cite{MR2815763}. We also use crucially a coupling by Lupu \cite{lupu2016loop}, which identifies the sign clusters of the GFF with the clusters of the BLS.

For simplicity, we restrict to the square lattice $\Z^2$, whose vertices are the points of $\R^2$ with integer coordinates (any two vertices being connected by an edge if they lie at Euclidean distance $1$). In this discrete setting, the RWLS with intensity $\alpha > 0$ is defined as a Poisson point process of finite loops on $\Z^2$ with intensity $\alpha \cdot \nu$, where $\nu$ is a natural (infinite) measure on random walk loops. It is known since \cite{SW2012} that the Brownian loop soup (BLS), which is the corresponding process in the continuum, displays a phase transition at the specific value $\alpha = 1/2$, in the following sense. Consider the BLS in a bounded domain. For each intensity $\alpha \leq 1/2$, the BLS contains almost surely (a.s.) infinitely many macroscopic connected components (or clusters) of loops; while for each $\alpha > 1/2$, there is only one such cluster of loops. Here (and in what follows), connectedness is defined in the natural way, i.e. two loops $\gamma$, $\gamma'$ are connected if one can go from $\gamma$ to $\gamma'$ by following a finite chain of loops (satisfying that any two successive loops in that chain overlap). At the critical intensity $\alpha = 1/2$, the case on which we focus in this paper, the macroscopic clusters ``almost touch''.

We also consider the BLS and the GFF on the associated metric graph (also known as cable system), that we denote by $\mZ$: this object is obtained by ``replacing'' each edge of the discrete graph $\Z^2$ with a line segment of length $1$. On $\mZ$, loop soup clusters correspond exactly to sign clusters of the GFF, which enables the use of tools for the GFF to deduce properties of the BLS.

Our main result in the present paper consists in up-to-constants estimates for arm events in the critical BLS on $\mZ$. More specifically, we consider the so-called four-arm events in the bulk, as well as the two-arm events in the half-plane. For these two families of events, we are able to obtain up-to-constants power law estimates, where the corresponding exponents are known to be, respectively, equal to $2$ and $1$. We state this result in an informal way now, and we refer the reader to Section~\ref{sec:setting} for precise definitions of the objects involved. For any integer $n \geq 0$, we introduce the notation $\partial {\wt B}_n :=\{v \in \mZ: \|v\|_{\infty} = n\}$, where $\|.\|$ denotes the usual $l_{\infty}$ norm on $\R^2$.

First, we consider the four-arm events. They are natural geometric events, used in particular in \cite{GNQ2} to quantify the fact that two distinct connected components of loops come close to each other. For any $1 \leq k\le n/2$, we denote by $\wt\pi_4(k,n)$ the probability that in ${\wt B}_{2n}$, there exist two outermost clusters of loops such that each of them intersects the two squares $\partial {\wt B}_k$ and $\partial {\wt B}_n$. 
\begin{theorem} \label{thm:main}
Consider the BLS in $\mZ$ with intensity $\alpha = 1/2$. There exist two absolute constants $c_1, c_2 > 0$ such that for all $1 \leq k\le n/2$,
\begin{align}\label{eq:four-arm}
c_1 \left(\frac{k}{n}\right)^2 \leq \wt\pi_4(k,n) \leq c_2 \left(\frac{k}{n}\right)^2.
\end{align}
\end{theorem}

Moreover, we also consider the restriction of the BLS to the upper half-plane $\Hb$ (obtained by simply discarding the loops that do not remain entirely in $\mZ\cap\Hb$). In that case, let $\wt\pi_2^+(k,n)$ be the probability of the event that in ${\wt B}_{2n} \cap \Hb$, there is one connected component of loops in $\Hb$ which intersects $\partial {\wt B}_k$ and $\partial {\wt B}_n$. This event also arises naturally in \cite{GNQ2}.

\begin{theorem} \label{thm:main2}
Consider the BLS in $\mZ \cap \Hb$ with intensity $\alpha = 1/2$. For all $1 \leq k\le n/2$,
\begin{align}\label{eq:two-arm}
c^+_1 \: \frac{k}{n} \leq \wt\pi_2^+(k,n) \leq c^+_2 \: \frac{k}{n},
\end{align}
where again $c^+_1, c^+_2 > 0$ are two absolute constants.
\end{theorem}



As a consequence of Theorems~\ref{thm:main} and~\ref{thm:main2}, we are able to derive the following upper bounds in the discrete setting (see \eqref{eq:arm_d} for the definitions of $\pi_2^+(k,n)$ and $\pi_4(k,n)$).

\begin{corollary}\label{cor:discrete-graph}
  Consider the RWLS in $\Z^2$ with intensity $\alpha=1/2$. There exist two absolute constants $c_3, c^+_3>0$ such that for all $1 \leq k\le n/2$,
\begin{align}
\label{eq:d-four-arm}
    &\pi_4(k,n)\le c_3 \left(\frac{k}{n}\right)^2,\\
    \label{eq:d-two-arm}
    &\pi_2^+(k,n)\le c^+_3 \: \frac{k}{n}.
\end{align}
\end{corollary}

Note that the RWLS on the discrete graph is stochastically dominated by the BLS on the corresponding metric graph, and the boundary two-arm event is increasing w.r.t.\ the loop soup. Hence, $\pi_2^+(k,n)\le \wt\pi_2^+(k,n)$, and \eqref{eq:d-two-arm} is a direct consequence of the upper bound from \eqref{eq:two-arm}, so that later we will only need to prove \eqref{eq:d-four-arm}. In contrast, the four-arm event is not monotone, and \eqref{eq:d-four-arm} cannot be obtained from \eqref{eq:four-arm} directly. Nevertheless, slightly modifying the proof for the upper bound of \eqref{eq:four-arm}, we can still get the desired upper bound \eqref{eq:d-four-arm}.

 The upper bounds in Corollary~\ref{cor:discrete-graph} allow us to verify Assumptions~5.7 and~5.11 in \cite{GNQ2}. Hence, from the reasonings in that paper, we can deduce in particular the existence of low crossings in the critical Random Walk Loop Soup, where the occupation field remains below $a \log \log N$, for a constant $a$ large enough (see Theorems~5.9 and~5.13 in \cite{GNQ2}, which were conditional results). In addition, Le Jan's isomorphism \cite{MR2815763} immediately yields the corresponding results for the discrete Gaussian free field (see Theorem~5.14 in the same paper), but now with level $a' \sqrt{\log \log N}$.


Finally, note that these arm exponents were in fact already derived in $\cite{GNQ1}$ (not only for $\alpha = 1/2$, but also for all subcritical intensities $\alpha < 1/2$), at least for the discrete graph (we have not attempted to show the corresponding results for the metric graph, but we expect that this case can be handled through similar arguments). There, the convergence of the RWLS to the BLS is used, combined with the connection between, on the one hand, outer boundaries of loop soup clusters in the BLS with intensity $\alpha$, and, on the other hand, the conformal loop ensemble (CLE) process with parameter $\kappa = \kappa(\alpha) \in (8/3,4]$. However, the bounds obtained there still contained an $o(1)$ in the exponent, i.e.\ potential logarithmic corrections, lost when transferring the continuous estimates for the BLS to the (discrete) RWLS. This was good enough for our applications in \cite{GNQ2} when $\alpha < 1/2$, since the corresponding arm exponents are then $>2$ and $>1$, respectively. But at criticality, stronger up-to-constants estimates are needed to get the more precise results, with a threshold of order $\log \log N$.

\section{Preliminaries} \label{sec:prelim}

\subsection{Setting} \label{sec:setting}

Recall that we denote the (two-dimensional) square lattice by $\Z^2$, and the metric graph associated with it by $\mZ$. For any $n \geq 0$, we let $B_n:=[-n,n]^2\cap\Z^2$ be the ball with radius $n$ (for $\|.\|_{\infty}$) centered on $0$, and for $0 \leq k \leq n$, we let $A_{k,n}:=\{v\in\Z^2: k \le \|v\|_{\infty} \le n\}$. We let $\Hb:=\{ (x,y)\in \R^2: y>0 \}$ be the upper half plane, and for a subset $A\subseteq\R^2$, we use the notation $A^+ := A\cap\Hb$ (resp. $A^- := A\cap(-\Hb)$). For $A\subseteq \Z^2$, we consider its metric graph version $\wt A$ in $\mZ$, where for every pair of vertices $v$, $v'$ in $A$ with Euclidean distance $1$, we introduce a line segment of length $1$ between $v$ and $v'$. For $z\in \R^2$ and $A\subseteq\R^2$, we introduce $\dist(z,A)=\inf_{w\in A} \|z-w\|_\infty$. 

For $A,B,F\subset\mZ$, we write
$$A\connect{F}B$$
to express one of the following facts: if $F\cap(\mZ\setminus \Z^2)\neq\emptyset$, then there exists a continuous path $\pi\subset F$ starting in $A$ and ending in $B$; otherwise, it means that there is a finite nearest neighbor path $(z_1,...,z_k)$ in $F$ with $z_1\in A$ and $z_k\in B$. When $F=\{z:\varphi(z)>0\}$, where $\varphi$ is a real-valued field defined on a subset of $\Z^2$ or $\mZ$, we simply write $A\connect{\varphi>0}B$. Moreover, we allow $F$ to be a collection of sets, in which case we will consider the union of the sets.

\subsubsection*{Random walk loop soup}
A \emph{rooted loop} $\gamma$ in $\Z^2$ is a path $(z_0,\cdots,z_j)$ in $\Z^2$ with some length $j \geq 0$, i.e.\ a sequence of vertices such that $z_k$ and $z_{k+1}$ are neighbors for each $k=0,\ldots,j-1$, which moreover satisfies $z_0=z_j$. In this case, we write $|\gamma| := j$. We also consider \emph{unrooted loops}, which are equivalence classes of rooted loops under rerooting (cyclic permutation of the vertices). In the remainder of the paper, we only consider unrooted loops, so in most instances we simply say ``loops''.

We introduce the measure $\nu$ on (unrooted) loops $\gamma$, defined by
\begin{equation} \label{eq:def_nu}
	\nu(\gamma):=\frac{4^{-|\gamma|}}{m_{\gamma}} \ind_{|\gamma|\geq 2},
\end{equation}
where $m_{\gamma}$ is the multiplicity of $\gamma$. For any $\alpha > 0$, the RWLS in $\Z^2$ with intensity $\alpha$, denoted by $\Lc$, is then simply the Poisson point process with intensity $\alpha\nu$.\footnote{Normally, one also needs to add in the trivial loops which consist of single points. For the isomorphism theorems that we mention later, the trivial loops are important, since they contribute to the occupation field of the RWLS.  However, in this paper, we are only interested in the arm events of the RWLS, hence do not need to consider the trivial loops.} In this paper, we only look at the $\alpha=1/2$ case. For a subset $D \subseteq \Z^2$, we use $\Lc_D$ to denote the corresponding loop soup on $D$, which is made of all the loops in the collection $\Lc$ that remain entirely in $D$ (that is, we only keep a loop $\gamma$ if all its vertices belong to $D$).
For brevity, we write $\Lc_n:=\Lc_{B_n}$.

\subsubsection*{Brownian loop soup on the metric graph}
A Brownian motion $X$ on the metric graph $\wt{\mathbb{Z}}^2$ moves like a standard one-dimensional Brownian motion on each edge (which is an interval of length $1$). At each vertex $z$, $X$ makes a Poisson point process of one-dimensional Brownian excursions, where each excursion has probability $1/4$ to live on one of the $4$ edges adjacent to $z$, and this process stops when $X$ hits a neighboring vertex of $z$ (i.e.\ the first time that an excursion reaches distance $1$ from $z$). Let $p_t(x,y)$ be the transition kernel of $X$. Let $\mathbb{P}^t_{x,y}(\cdot)$ be the probability measure of a Brownian bridge from $x$ to $y$ with duration $t$.

A \emph{rooted loop} $\wt\gamma$ in $\wt{\mathbb{Z}}^2$ is a curve $(\gamma(t), 0\le t\le T)$ with some length $T>0$ and $\gamma(0)=\gamma(T)$, such that $\gamma(t)\in \wt{\mathbb{Z}}^2$ for all $0\le t\le T$. We can also unroot a loop by forgetting its starting point. The Brownian loop measure on $\wt{\mathbb{Z}}^2$ is a measure on unrooted loops given by
\[
\wt\nu(\wt\gamma):=\int_{z\in \wt{\mathbb{Z}}^2}\int_0^\infty \mathbb{P}^t_{z,z}(\wt\gamma) p_t(z,z) \frac{dt}{t} dz,
\]
where $dz$ is according to the Lebesgue measure on $\wt{\mathbb{Z}}^2$.
For any $\alpha>0$, the BLS in $\mZ$ with intensity $\alpha$, denoted by $\wt\Lc$, is the Poisson point process with intensity $\alpha \wt\nu$. We set $\alpha=1/2$ throughout. For a subset $\wt D\subseteq \mZ$, we use $\wt\Lc_{\wt D}$ to denote the BLS on $\wt D$, which is the collection of loops in $\wt \Lc$ that stay entirely in $\wt D$. For brevity, we write $\wt\Lc_n:=\wt\Lc_{\wt B_n}$.

For $\wt D\subset \mZ$, let $\mathrm{int} (\wt D)$ be the interior of $\wt D$ in $\mZ$. Let $\partial \wt D$ be the closure of $\wt D$ minus $\mathrm{int} (\wt D)$. For $x\in \mathrm{int}(\wt D)$, the Brownian motion on $\mZ$ started from $x$ induces a probability measure on Brownian excursions started from $x$ stopped upon hitting $\partial \wt D$. The restriction of this probability measure on excursions ending at $y$ for each $y\in \partial \wt D$ induces a (non probability) measure $\wt \nu^{\wt D}_{x, y}$.
For $x\in \partial \wt D$, let $x^1_\eps, \ldots, x^k_\eps$ be the points in $\partial \wt D$ which have distance $\eps$ to $x$.
For $x, y \in \partial D$, we define the boundary-to-boundary excursion measure 
\begin{align}\label{eq:exc}
    \wt \nu_{x,y}^{\wt D}=\lim_{\eps\to 0} \eps^{-1}\sum_{j=1}^k \wt\nu^{\wt D}_{x^j_\eps, y}.
\end{align}


\subsubsection*{Arm events in a loop soup}
Now, we follow the notion of arm events introduced in \cite{GNQ1,GNQ2}. 

\begin{definition} \label{def:2n arm event for rwls}
	Let $1 \leq k<n$ and $D\subseteq \Z^2$. 
    The \emph{interior four-arm event} in the annulus $A_{k,n}$ and the \emph{boundary two-arm event} in the semi-annulus $ A^+_{k,n}$, are respectively defined as
\begin{itemize}
\item $\Ac_{D}(k,n) := \{$there are at least two outermost clusters in $\Lc_{D}$ crossing $ A_{k,n}\}$
\item $\Bc_{D}(k,n) := \{$there is a cluster in $\Lc_{D^+}$ crossing $ A^+_{k,n}\}$.
\end{itemize}
The \emph{local arm events} are defined by $\Ac_{\loc}(k,n):=\Ac_{ B_{2n}}(k,n)$ and $\Bc_{\loc}(k,n):=\Bc_{ B_{2n}}(k,n)$. 

We define the corresponding arm events on the metric graph in the same way, and use the symbol ``$\sim$'' above the notation to distinguish this case. For instance, $\wt\Ac_{\wt D}(k,n)$ denotes the event that there are at least two outermost clusters in $\wt\Lc_{\wt D}$ crossing $\wt A_{k,n}$, and $\wt\Ac_{\loc}(k,n):=\wt\Ac_{\wt B_{2n}}(k,n)$.
\end{definition}

In this paper, we investigate the asymptotics of the following four types of arm events: 
\begin{align}
\label{eq:arm_d}
&\pi_2^+(k,n):=\Pb[ \Bc_{\loc}(k,n) ], \quad \pi_4(k,n):=\Pb[ \Ac_{\loc}(k,n) ],\\
\label{eq:amr_m}
&\wt\pi_2^+(k,n):=\Pb[ \wt\Bc_{\loc}(k,n) ], \quad
\wt\pi_4(k,n):=\Pb[ \wt\Ac_{\loc}(k,n) ].
\end{align}
When $k=1$, we simply drop $k$ from the notation.





\subsubsection*{Isomorphism theorems}

Le Jan \cite{MR2815763} proved in a rather general setting that the occupation field of a loop soup (either the BLS in $\mathbb{R}^d$ for $d=1,2,3$ or a discrete RWLS on a general graph) in a domain with non-polar boundary is equal to $1/2$ times the square of a GFF with zero boundary conditions in the same (continuous or discrete) domain.
This is the loop-soup version of a family of isomorphism theorems, dating back to
\cite{Symanzik66Scalar,Symanzik1969QFT,
	BFS82Loop,Dynkin1984Isomorphism,
	Dynkin1984IsomorphismPresentation,
	Wolpert2,
	EKMRS2000,
	MarcusRosen2006MarkovGaussianLocTime}.


In \cite{lupu2016loop}, Lupu first looked at the case of a BLS on the metric graph, which turns out to be instrumental in numerous circumstances. In particular, Lupu constructed the following coupling between the BLS and the GFF $\varphi$ on the same metric graph: Given a BLS on a metric graph $\wt D$ and its occupation field $\psi$, 
one can obtain a GFF $\varphi$ on $\wt D$ by taking the square root of $\psi$, and then giving i.i.d.\ signs $\pm 1$ to each cluster of the BLS. In this way, the sign clusters of $\varphi$ are exactly the same as the clusters of the BLS.

The GFF's in \cite{MR2815763} and \cite{lupu2016loop} have zero boundary conditions. There is also a version of isomorphism for GFF with positive boundary conditions, due to Sznitman \cite{MR2892408,MR3167123}(also see \cite{MR2932978}), which also works in a general setting, for both the BLS and the RWLS. More precisely, it states that for such a GFF in $D$, its square times $1/2$ is equal to the occupation field of a loop soup in that domain, plus that of a Poisson point process of boundary-to-boundary excursions with an appropriate intensity. 
This version of isomorphism combined with the ideas of the metric graph \cite{lupu2016loop} immediately implies the following coupling (see e.g.\ \cite[Proposition~2.4]{aru2020first}). 
\begin{lemma}\label{lem:iso}
Fix a metric graph $\wt D\subset \mZ$. Let $\varphi$ be a GFF on $\wt D$ with boundary conditions $u(y)\ge 0$ on each $y\in \partial \wt D$. Let $\wt {\mathcal{E}}$ be a Poisson point process of excursions with intensity $\frac12 \sum_{x,y \in \partial \wt D} u(x) u(y) \wt \nu_{x,y}^{\wt D}$.
One can couple $\varphi$ with $(\wt \Lc_{\wt D}, \wt{\mathcal{E}})$ in a way that the occupation field of the latter is equal to $\varphi^2/2$ and the clusters of $\wt\Lc \cup \wt{\mathcal{E}}$ are the same as the sign clusters of $\varphi$.
\end{lemma}

\section{Boundary two-arm event: Proof of (\ref{eq:two-arm})}
In this section, we present the proof of \eqref{eq:two-arm}. We use the isomorphism theorem to reduce the upper bound to an estimate on a certain connection probability of the GFF, which can be deduced from \cite[Corollary 1]{lupu2018random}. Meanwhile, we show the lower bound through the quasi-multiplicativity of arm events from \cite[Proposition 6.3]{GNQ1}.


For any integer $k\in[1,n/2]$, let $\mphi^{k}_n$ be the metric-graph GFF on $\mB_{2n}^+$ with boundary conditions
$$\mphi^{k}_n(z)=\begin{cases}
    1,&z\in l_{3k/2};\\
    0,&z\in \partial(\mB_{2n}^+)\setminus l_{3k/2},
\end{cases}$$
where $l_{3k/2}:=\left[-\lfloor\frac{3k}{2}\rfloor,\lfloor\frac{3k}{2}\rfloor\right]$. 
Then, we have the following up-to-constants estimate for the connection probability of GFF from the segment $l_{3k/2}$ to $(\partial\mB_n)^+$, which will be a key ingredient in the proof of the upper bound of (\ref{eq:two-arm}).
\begin{lemma}\label{le:gff}
For all $1 \leq k\le n/2$,
    $$\mathbb P[l_{3k/2}\connect{\mphi^{k}_n>0}(\partial\mB_n)^+]\asymp\frac kn.$$
\end{lemma}

\begin{proof}
For any $x\in\wt{\mathbb Z}^2$ and any nonempty subset $A\subset\wt{\mathbb Z}^2$, let $X$ under $P^x$ denote the Brownian motion started from $x$ on $\wt{\mathbb Z}^2$, and let $H_A$ be the first hitting time of $A$ by $X$. Let $R^{\rm eff}(x,A)$ be the effective resistance from $x$ to $A$ in $\wt{\mathbb Z}^2$ and let $G_A$ denote the Green's function for the Brownian motion $X$ killed upon hitting $A$, then by definition 
\begin{equation}\label{eq:RGnew}
    R^{\rm eff}(x,A)=G_A(x,x).
\end{equation}
The geometric considerations in this proof are illustrated on Figure~\ref{fig:connection_GFF}.

We first deal with the lower bound.
Let $x_0=\frac{3n}2 i$.
Let $\widetilde\Lambda_0$ be the union of clusters in $\{u\in\mB_{2n}^+:\mphi^k_{n}(u)>0\}$ that intersect $l_{3k/2}$.
By (\ref{eq:RGnew}) and \cite[Lemma 2.5]{ding2022crossing}
here exists a constant $c>0$ such that 
    \begin{equation}\label{eq:sandwich}
        \{R^{\rm eff}(x_0,\partial (\mB_{2n}^+))-R^{\rm eff}(x_0,\widetilde\Lambda_0\cup\partial(\mB_{2n}^+))>c\}\subset\{l_{3k/2}\connect{\mphi^{k}_n>0}(\partial\mB_n)^+\}.
    \end{equation}
    Note that
    $$\mathbbm E[\mphi^{k}_n(x_0)]=P^{x_0}[X_{H_{\partial(\mB_{2n}^+)}}\in l_{3k/2}]\asymp\frac kn,$$
    where the asymptotic relation follows from \cite[Proposition 8.1.3]{lawler2010random}. Thus, by \cite[Corollary 1]{lupu2018random}, we have
    \begin{equation}\label{eq:brownian}
        \mathbb P[R^{\rm eff}(x_0,\partial (\mB_{2n}^+))-R^{\rm eff}(x_0,\widetilde\Lambda_0\cup\partial(\mB_{2n}^+))>c]=\mathbb P\left[W_t<\mathbbm E[\mphi^{k}_n(x_0)],\ \forall0\le t\le c\right]\asymp\frac kn,
    \end{equation}
    where $W$ is the one-dimensional standard Brownian motion.
    Hence, combining (\ref{eq:sandwich}) and (\ref{eq:brownian}), we have
    \begin{equation}\label{eq:llower}
        \mathbb P[l_{3k/2}\connect{\mphi^{k}_n>0}(\partial\mB_n)^+]\gtrsim\frac kn.
    \end{equation}

    Next, we turn to the upper bound. In this case, we consider the GFF in a larger region with some different boundary conditions. 
    Specifically, let $\mphi_n^{k,*}$ be the GFF on $D^*:=\mB_{2n}^+\cup\mB_{2n}^-\cup[-9n/8,-7n/8]\cup[7n/8,9n/8]$ with boundary conditions
$$\mphi_n^{k,*}(z)=\begin{cases}
    1,&z\in l_{3k/2};\\
    0,&z\in\partial D^*\setminus l_{3k/2}.
\end{cases}$$
Note that compared with $\mphi_n^{k}$, we do not put any boundary conditions for $\mphi_n^{k,*}$ on the segment $[-9n/8,-7n/8]\cup[7n/8,9n/8]$, and extend the region on with it is defined to contain the reflected box $\mB_{2n}^-$.
Now, we claim that 
\begin{equation}\label{eq:mphi*}
    \Pb[l_{3k/2}\connect{\mphi_n^k>0}(\partial \mB_n)^+]\le\Pb[l_{3k/2}\connect{\mphi_n^{k,*}>0}(\partial \mB_n)^+].
\end{equation}
To see it, we let $\wt\Lc$ be the loop soup in $\mB_{2n}^+$, and $\wt\Lc^*$ be the loop soup in $D^*$, respectively. Note that $\wt\Lc^*$ contains loops crossing the real axis through line segments $[-9n/8,-7n/8]$ and $[7n/8,9n/8]$ and loops staying in $\mB_{2n}^-$, while $\wt\Lc$ does not, and they can be coupled such that $\wt\Lc\subseteq \wt\Lc^*$.
Furthermore, we use $\mathcal E$ and $\mathcal E^*$ to denote the set of boundary-to-boundary excursions on $\mB_{2n}^+$ and $D^*$ with intensity $\frac12 \sum_{x,y \in l_{3k/2}} \wt\nu_{x,y}^{\mB_{2n}^+}$ and $\frac 12 \sum_{x,y \in l_{3k/2}}\wt\nu_{x,y}^{D^*}$ respectively (see \eqref{eq:exc} for a definition), independent of the loop soups $\wt\Lc$ and $\wt\Lc^*$. Then the two excursion clouds can be similarly coupled so that $\mathcal E\subseteq\mathcal E^*$ almost surely. Then, by the isomorphism theorem between the GFF and loop soups (Lemma~\ref{lem:iso}), we can rewrite the connection probability for the GFF in terms of the loop soups together with the excursions:
    \begin{equation}\label{eq:iso}
        \mathbb P[l_{3k/2}\connect{\mphi^{k}_n>0}(\partial\mB_n)^+]=\mathbb P[l_{3k/2}\connect{\wt\Lc\oplus\mathcal E}(\partial\mB_n)^+], 
    \end{equation}
    and 
    \[
    \mathbb P[l_{3k/2}\connect{\mphi^{k,*}_n>0}(\partial\mB_n)^+]=\mathbb P[l_{3k/2}\connect{\wt\Lc^*\oplus\mathcal E^*}(\partial\mB_n)^+],
    \]
    where $\oplus$ stands for the merging of two point processes.
    Therefore, \eqref{eq:mphi*} follows from the equations above, since $\wt\Lc\subseteq \wt\Lc^*$ and $\mathcal E\subseteq\mathcal E^*$.

It remains to bound the RHS of \eqref{eq:mphi*}.
Similarly, we denote by  $\widetilde\Lambda_0^*$ the union of clusters in $\{u\in\mB_{2n}^+:\mphi_n^{k,*}(u)>0\}$ that intersect $l_{3k/2}$.
    In this case, we claim that there exists a constant $c'>0$ such that
    \begin{equation}\label{eq:reff}
        \{l_{3k/2}\connect{\mphi_n^{k,*}>0}(\partial \mB_n)^+\}\subset\{R^{\rm eff}(x_0,\partial D^*)-R^{\rm eff}(x_0,\partial D^*\cup\widetilde\Lambda_0^*)>c'\}.
    \end{equation}
Let us prove \eqref{eq:reff}. 
    By (\ref{eq:RGnew}),
    \begin{equation}\label{eq:RG}\begin{aligned}
        R^{\rm eff}(x_0,\partial D^*)-R^{\rm eff}(x_0,\partial D^*\cup\widetilde\Lambda_0^*)&=G_{\partial D^*}(x_0,x_0)-G_{\partial D^*\cup\widetilde\Lambda_0^*}(x_0,x_0)\\
        &= \sum_{u\in\widetilde\Lambda_0^*\setminus \partial D^*}P^{x_0}[X_{H_{\partial D^*\cup\widetilde\Lambda_0^*}=u}]\cdot G_{\partial D^*}(u,x_0).
    \end{aligned}\end{equation}
    We assume that the event $\{l_{3k/2}\connect{\mphi_n^{k,*}>0}(\partial \mB_n)^+\}$ occurs, and we take an arbitrary point $z\in\widetilde\Lambda_0^*\cap(\partial\mB_n)^+$. Then, it is easy to see that there exist universal constants $c_1,c_2>0$ such that 
    \[
    P^{x_0}[H_{\widetilde\Lambda_0^*\cap\mB_{n/100}(z)}<H_{\partial D^*}]>c_1, \quad \text{and }G_{\partial D^*}(u,x_0)>c_2 \text{ for all } u\in \mB_{n/100}(z).
    \]
    Therefore, on $\{l_{3k/2}\connect{\mphi_n^{k,*}>0}(\partial \mB_n)^+\}$, by (\ref{eq:RG}),
    \begin{equation}\label{eq:Rc}\begin{aligned}
        R^{\rm eff}(x_0,\partial D^*)-R^{\rm eff}(x_0,\partial D^*\cup\widetilde\Lambda_0^*)\ge&\ \sum_{u\in\widetilde\Lambda_0^*\cap\mB_{n/100}(z)}P^{x_0}[X_{H_{\partial D^*\cup\widetilde\Lambda_0^*}=u}]\cdot c_2\\\ge&\  c_2\,P^{x_0}[H_{\widetilde\Lambda_0^*\cap\mB_{n/100}(z)}<H_{\partial D^*}]>c_1c_2,
    \end{aligned}\end{equation}
    yielding (\ref{eq:reff}).
    
    Next, we show the following estimate 
    \begin{equation}\label{eq:*kn}
        \mathbb E[\mphi_n^{k,*}(x_0)]\asymp\frac kn.
    \end{equation}
    Let $l'=(i+l_{3k/2})\cup(-i+l_{3k/2})$. Using the last exit decomposition of the Brownian motion, we deduce that
    $$\mathbb E[\mphi_n^{k,*}(x_0)]=P^{x_0}[X_{H_{\partial D^*}}\in l_{3k/2}]\le\sum_{u\in l'\cap \mathbb Z^2}G_{\partial D^*}(x_0,u),$$
    and for each $u\in l'\cap \Z^2$,
    $$G_{\partial D^*}(x_0,u)\le P^u[H_{\partial \mB_{n/100}(u)}<H_{\mathbb R}]\cdot\sup_{z\in\partial\mB_{n/100}(u)}G_{\partial D^*}(z,x_0)\lesssim \frac 1n,$$
    where in the last inequality, we use $P^u[H_{\partial \mB_{n/100}(u)}<H_{\mathbb R}]\lesssim\frac1n$ from \cite[Lemma 2.6]{ding2022crossing}, and
    $G_{\partial D^*}(z,x_0)\le G_{\partial\mB_{2n}}(z,x_0)\lesssim 1$
    from \cite[Lemmas 2.9 and 2.13]{chelkak2016robust}. Therefore, we conclude \eqref{eq:*kn}.
    Now, combining \eqref{eq:reff} and \eqref{eq:*kn}, and applying \cite[Corollary 1]{lupu2018random} again, we obtain that
    \begin{equation}\label{eq:lupper}
    \begin{aligned}
        \Pb[l_{3k/2}\connect{\mphi_n^{k,*}>0}(\partial \mB_n)^+]\le&\ \Pb[R^{\rm eff}(x_0,\partial D^*)-R^{\rm eff}(x_0,\partial D^*\cup\widetilde\Lambda_0^*)>c']\\
        \le&\ \mathbb P\left[W_t<\mathbbm E[\mphi^{k,*}_n(x_0)],\ \forall 0<t<c'\right]\asymp\frac kn.
    \end{aligned}
    \end{equation}
    This completes the proof of the lemma by \eqref{eq:mphi*}.
\end{proof}

\begin{figure}[h]
\centering
\includegraphics[width=0.8\linewidth]{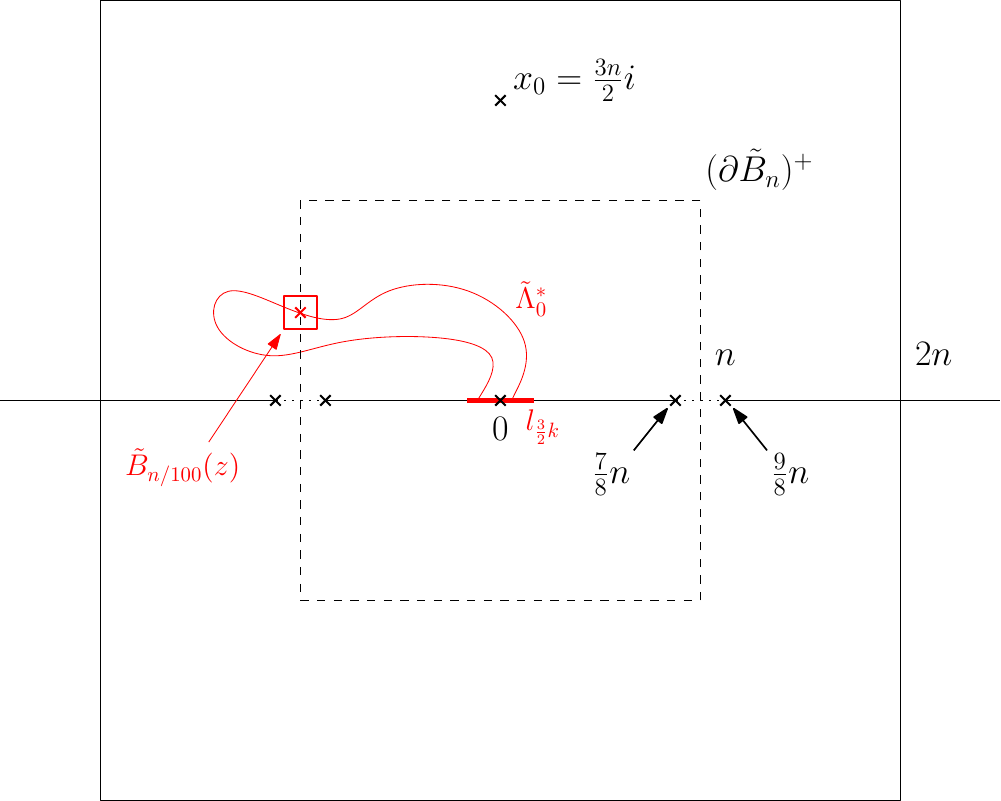}
\caption{\label{fig:connection_GFF}Proof of Lemma~\ref{le:gff}.}
\end{figure}

With the above lemma, we readily obtain the upper bound of (\ref{eq:two-arm}).
\begin{proof}[Proof of the upper bound of (\ref{eq:two-arm})]
    We follow the notation used in the proof of Lemma \ref{le:gff}.
    By \eqref{eq:iso} and Lemma \ref{le:gff}, we obtain that 
    \begin{equation}\label{eq:iso1}
        P[l_{3k/2}\connect{\wt\Lc\oplus\mathcal E}(\partial\mB_n)^+]\asymp\frac k n.
    \end{equation}
    Let $E$ be the event that there exists an excursion in $\mathcal E$ disconnecting $\mB_k^+$ from $(\partial\mB_{n})^+$ inside $\mB_{2n}^+$. Since the total mass of such excursions is bounded away from $0$ uniformly in $k\in [1, n/2]$ and $n$, there is a universal constant $c>0$ such that $\mathbb P[E]>c$. 
    Since the loop soup $\wt\Lc$ and the excursions $\mathcal{E}$ are independent, we have
    \begin{equation}\label{eq:exc-disc}
        \mathbb P[l_{3k/2}\connect{\wt\Lc\oplus\mathcal E}(\partial\mB_n)^+]\ge\mathbb P[E\cap\widetilde{\mathcal B}_{\rm loc}(k,n)]=\mathbb P[E]\cdot\mathbb P[\widetilde{\mathcal B}_{\rm loc}(k,n)]\ge c\,\mathbb P[\widetilde{\mathcal B}_{\rm loc}(k,n)].
    \end{equation}
    It then follows from (\ref{eq:iso1}) that
    \begin{equation}\label{eq:twoarm_upper}
        \mathbb P[\widetilde{\mathcal B}_{\rm loc}(k,n)]\lesssim \frac kn,
    \end{equation}
    which finishes the proof of the upper bound.
    \end{proof}

Now we turn to the lower bound of (\ref{eq:two-arm}). For this, we first derive the corresponding lower bound for $\mathbb P[\mB_{\rm loc}(1,n)]$.

\begin{lemma}\label{lem:low1}
    Let $\mphi_n$ be the metric-graph GFF on $\mB_{2n}^+$ with zero boundary conditions. For all $n\ge 2$,
    \begin{equation}
        \mathbb P[\mB_{\rm loc}(1,n)]\ge\mathbb P[i\connect{\mphi_n>0}(\partial \mB_{n})^+]\gtrsim \frac1n.
    \end{equation}
\end{lemma}
\begin{proof}
    By Lemma~\ref{lem:iso}, $\mathbb P[\mB_{\rm loc}(1,n)]$ is equal to the probability there exists a sign (plus or minus) cluster of $\mphi_n$ that crosses $A^+_{1,n}$. Hence,
    the first inequality follows immediately. 
    For the second one, note that 
    \[
    \mathbb P[i\connect{\mphi_n>0}(\partial \mB_{n})^+]\ge 
    \mathbb P[i\connect{\mphi_n>0}(\partial \mB_{n})^+\mid\mphi_n(i)\ge1]\cdot \mathbb P[\mphi_n(i)\ge1] \gtrsim \mathbb P[i\connect{\mphi_n>0}(\partial \mB_{n})^+\mid\mphi_n(i)\ge1].
    \]
    Moreover, by \cite[Lemma 4.2]{ding2022crossing},
    \[
    \mathbb P[i\connect{\mphi_n>0}(\partial \mB_{n})^+\mid\mphi_n(i)\ge1] \ge
    \mathbb P[i\connect{\mphi_n>0}(\partial \mB_{n})^+\mid\mphi_n(i)=1].
    \]
    Using an argument similar to the proof of Lemma \ref{le:gff} (reducing the segment $l_{3k/2}$ to the single boundary point $i$), we obtain that
    $$\mathbb P[i\connect{\mphi_n>0}(\partial \mB_{n})^+\mid\mphi_n(i)=1]\asymp\frac1n.$$
    Combining all the above estimates, we conclude the proof.
\end{proof}
    


\begin{proof}[Proof of the lower bound of (\ref{eq:two-arm})]
By the quasi-multiplicativity of arm events in \cite[Proposition~6.3]{GNQ1}, we have
    \begin{equation}\label{eq:quasi}
        \mathbb P[\mB_{\rm loc}(1,n)]\lesssim \mathbb P[\mB_{\rm loc}(1,k)]\cdot\mathbb P[\mB_{\rm loc}(k,n)]
    \end{equation}
Using Lemma~\ref{lem:low1} and the upper bound of \eqref{eq:two-arm} on the above two sides respectively, we obtain that
\[
\frac1n \lesssim \mathbb P[\mB_{\rm loc}(1,n)]\lesssim \mathbb P[\mB_{\rm loc}(1,k)]\cdot\mathbb P[\mB_{\rm loc}(k,n)] \lesssim \frac1k \cdot \mathbb P[\mB_{\rm loc}(k,n)],
\]
    which yields the desired lower bound.
\end{proof}

\section{Four arm event}

In this section, we complete the proofs of \eqref{eq:four-arm} and \eqref{eq:d-four-arm}, respectively in Sections~\ref{subsec:m_four_arm} and \ref{subsec:d_four_arm}. Before that, in Section~\ref{subsec:toolbox}, we collect useful properties established in earlier works \cite{GNQ1,GNQ3} for four-arm events. In fact, these papers are written for the discrete graph $\Z^2$, but we will also use the corresponding results for the metric graph without proof, since they can readily be obtained by adapting the arguments from the above-mentioned papers.

\subsection{A toolbox}\label{subsec:toolbox}
We first present two useful tools for four-arm events that have been obtained in \cite{GNQ1,GNQ3}, namely, the separation lemma and the quasi-multiplicativity property. Each of these results holds for a general class of setups, by following a proof along similar lines. In particular, they hold for four-arm events in loop soups on discrete graphs, on metric graphs, or even in a continuum domain (i.e.\ the Brownian loop soup). For applications in this paper, we state some versions of these properties for loop soups on the discrete graph and on the metric graph, and we refer the reader to \cite{GNQ1,GNQ3} for more details.

\begin{lemma}[Separation lemma at both sides]\label{lem:sep}
    Let $1\le k\le n/2$. Define the separation event $\Ac_{\loc}^{\rm{sep}}(k,n)$ which satisfies the following two conditions:
    \begin{itemize}
        \item There are only two outermost clusters $\cC_1$ and $\cC_2$ in $\Lc_{B_{2n}}$ crossing  $A_{k,n}$;
        \item $\cC_j \subset (A_{k,n} \cup B_{n/20}(ne^{i(j-1)\pi}) \cup B_{k/20}(ke^{i(j-1)\pi})) \setminus (B_{n/10}(ne^{i j \pi})\cup B_{k/10}(ke^{i j \pi}))$ for both $j=1,2$.
    \end{itemize}   
The separation event on the metric graph is defined similarly and denoted by $\wt\Ac_{\loc}^{\rm{sep}}(k,n)$.
Then, we have 
\[
\Pb[ \Ac_{\loc}^{\rm{sep}}(k,n) ] \asymp \pi_4(k,n), \quad \Pb[ \wt\Ac_{\loc}^{\rm{sep}}(k,n) ] \asymp \wt\pi_4(k,n).
\]
\end{lemma}
\begin{proof}
    A version of this separation lemma at the outer side $n$ for the Brownian loop soup has been proved in \cite[Theorem~1.2]{GNQ3}. In fact, one can generalize this result to the case where separation occurs at both sides, using a continuous version of \cite[Proposition~4.8]{GNQ1}. 
    To obtain this result for the RWLS on a discrete graph, one can follow the same proof as in \cite[Theorem~1.2]{GNQ3}, using results in \cite{GNQ1} as inputs. It is easy to obtain the analogous results for loop soups on the metric graph as well.  
\end{proof}

\begin{lemma}[Quasi-multiplicativity]\label{lem:quasi}
     For all $1\le k\le n/2$, we have	\begin{equation}\label{eq:quasi-1}
	\pi_4(n) \asymp \pi_4(k)\, \pi_4(k,n), \quad \wt\pi_4(n) \asymp \wt\pi_4(k)\, \wt\pi_4(k,n).
	\end{equation}
\end{lemma}
\begin{proof}
    We focus on the discrete estimate as the metric-graph one follows similarly. The upper bound $\pi_4(n) \lesssim \pi_4(k)\, \pi_4(k,n)$ has been shown in \cite[Proposition~6.1]{GNQ1}. The lower bound now follows from Lemma~\ref{lem:sep} immediately by basic geometric constructions.
\end{proof}

\begin{lemma}\label{lem:qua1}
    For all $\delta_1\in (0,1]$ and $\delta_2\ge 1$, there exists $c\in (0,1)$ such that for all $1\le k \le n/2$,
    \[
    c\,\pi_4(k,n)\le  \pi_4(\delta_1 k,\delta_2 n) \le c^{-1}\pi_4(k,n).
    \]
\end{lemma}
\begin{proof}
    By Lemmas~\ref{lem:sep} and~\ref{lem:quasi}, we only need to show $\pi_4(n,\delta_2 n)\ge c(\delta_2)>0$. Note that $\pi_4(n,\delta_2 n)$ converges to some four-arm probability for the Brownian loop soup as $n\to\infty$ (see e.g.\ \cite[(6.16)]{GNQ1}), which is positive by \cite[Theorem~1.1]{GNQ3}. This gives the result.
\end{proof}

\begin{remark}
    Note that all the above results, i.e. Lemmas~\ref{lem:sep}, ~\ref{lem:quasi} and~\ref{lem:qua1}, are not specific to the critical intensity $\frac12$: they hold for four-arm events in loop soups with any intensity in  $(0,\frac12]$.
\end{remark}

\subsection{Proof of \eqref{eq:four-arm}}\label{subsec:m_four_arm}
In order to prove \eqref{eq:four-arm}, it is enough to prove the following estimate for $\wt\pi_4(n)=\wt\pi_4(1,n)$.
\begin{proposition}\label{prop:k=1}
    We have
    \begin{equation}\label{eq:four-arm1}
        \wt\pi_4(n)\asymp n^{-2},
    \end{equation}
    where the implied constants do not rely on $n$.
\end{proposition}

\begin{proof}[Proof of \eqref{eq:four-arm} given Proposition~\ref{prop:k=1}]
    We obtain \eqref{eq:four-arm} directly by combining Proposition~\ref{prop:k=1} and Lemma~\ref{lem:quasi}.  
\end{proof}

Before diving into the proof of Proposition~\ref{prop:k=1}, we collect some basic results on the loop soups.
Recall that $\Lc_k$ is the random walk loop soup on the discrete box $B_k$ with intensity $\frac12$, which can be obtained from the continuous loops in $\wt\Lc_k=\wt\Lc_{\wt B_k}$ by taking the print of the latter on $B_k$ (see e.g.\ \cite[Section~2]{lupu2016loop}). The next lemma shows that there exist surrounding loops in annuli with uniformly positive probability for both loop soups.

\begin{lemma}\label{lem:loop}
Let $H_{k}$ be the event that there exists a loop in $\Lc_k$ that surrounds $B_{k/2}$,
and let $\wt H_{k}$ be the corresponding event for $\wt\Lc_k$. There is $c>0$ such that for all $k\ge 1$, 
    \[
    \Pb[H_{k}]\ge \Pb[\wt H_{k}]>c.
    \]
\end{lemma}
\begin{proof}
    The first inequality holds by definition.
    Hence, it suffices to show that $\Pb[\wt H_k]> c$ when $k$ is sufficiently large (the probability is uniformly positive for all small $k$). Now, by \cite[Theorem~1.1]{lawler2007random}, we have $\lim_{k\to\infty}\Pb[\wt H_k]=\lim_{k\to\infty}\Pb[ H_k]=a$, where $a$ is the probability that there exists a loop in the Brownian loop soup in $\D_{1}$ with intensity $\frac12$ that surrounds $\D_{1/2}$. Here, $\D_r:=\{x\in\R^2: \|x\|_\infty\le r\}$. Finally, we have $a>0$ by \cite[Proposition 3.9]{lawler2011defining}, which implies that $\Pb[\wt H_k]\ge \frac a2$ for sufficiently large $k$. This finishes the proof.
\end{proof}

Of course, one can replace the box $B_{k/2}$ in the statement of Lemma~\ref{lem:loop} by $B_{\delta k}$ for any $\delta\in (0,1)$ such that the lower bound $c$ depends only on $\delta$.
Next, we use it to construct some clusters in need later. Here and below, the \emph{outer boundary} of a connected set $\wt A$ in $\mZ$, denoted by $\partial_\out \wt A$, refers to the connected component of $\partial \wt A$ that can be connected to infinity in $\wt A^c$.

\begin{lemma}\label{lem:cluster}
    Let $F_n$ be the event that there exists a cluster of $\Lc_{n}$ that intersects $B_{n/8}$ with outer boundary contained in $A_{n/8,n/4}$, and let $\wt F_n$ be the corresponding event for $\wt\Lc_n$. Then, there is a universal constant $c$ such that $\min\{\Pb[F_n],\Pb[\wt F_n]\}>c$.
\end{lemma}

\begin{proof}
    In the proof, we view $\Lc_r$ as $\Lc_{n}$ restricted to loops staying inside $ B_r$ when $r<n$. 
    For the lower bound of $\Pb[ F_n]$, we define the following three events
    \begin{itemize}
        \item $ E_1(n)$: there exists a loop of $\Lc_{n/7}$ that  surrounds $ B_{n/8}$;
        \item $ E_2(n)$: there exists a loop of $\Lc_{n/6}$ that crosses $ A_{n/8,n/7}$;
        \item $ E_3(n)$: all clusters in $\Lc_n\setminus \Lc_{n/6}$ have diameter smaller than $n/12$.
    \end{itemize}
    Note that the above three events are independent (involving disjoint sets of loops), and $E_1(n)\cap  E_2(n) \cap  E_3(n)\subset  F_n$. Hence, we only need to bound the probability of each event $E_j(n)$ from below by some universal constant. For $j=1,2$, this can be deduced from Lemma~\ref{lem:loop} and the restriction property of the loop soup. The case $j=3$ follows from \cite[Lemma~3.2]{GNQ1}.
    Therefore, we have $\Pb[F_n]>c$.
    We conclude the proof as the same argument works for the metric graph loop soup as well.
\end{proof}

Using an argument for the random cluster model with weight $q=4$ from~\cite{duminil2022planar}, we are able to derive an up-to-constants estimate for the quantity $\Nc(\Lambda)$ below, which will be crucial in our analysis later.


\begin{lemma}\label{lem:NG}
    Let $\Lambda$ be a metric graph such that $\wt B_{n/8}\subset \Lambda \subset \wt B_{n/4}$. Define the outer boundary  $\Gamma:=\partial_\out\{z\in \Z^2: \dist(z,\Lambda)\le 2\}$. 
    Let $\Nc(\Lambda)$ be the mean number of points in $\Gamma$ that is connected to $\partial\mB_{n/2}$ outside $\Lambda$, that is, 
    \begin{equation}\label{eq:Nc}
    \Nc(\Lambda):=\sum_{u\in\Gamma} \mathbb P[u\connect{\wt\Lc_{\mB_n\setminus \Lambda}}\partial\mB_{n/2}]. 
    \end{equation}
    Then, there exist universal constants $0<c<C<\infty$ such that $c<\mathbb \Nc(\Lambda)<C$.
\end{lemma}
\begin{proof}
For any $n\in\mathbb N^*$ and any $A\subset\mB_n$, $u\in\mB_n\setminus A$, $v\in A$, define the harmonic measure ${\rm Hm}^n(u,v;A)$ by
\begin{equation}\label{eq:def_hm}
    {\rm Hm}^n(u,v;A):=P^u[X_{H_{A\cup\partial\mB_n}}=v],
\end{equation}
where $X$ is the Brownian motion on $\wt\Z^2$ started from $u$, and $H_{A\cup\partial\mB_n}$ is the first hitting time of $A\cup\partial\mB_n$ by $X$. 
Let ${\rm Hm}^n(u,A):=\sum_{v\in A}{\rm Hm}^n(u,v;A)$.
We abbreviate ${\rm Hm}={\rm Hm}^n$ below for brevity. 

Let  $x_0=\frac34ni$. We first show that for all $u\in\Gamma$,
\begin{equation}\label{eq:two-arm'}
        \mathbb P[u\connect{\wt{\mathcal L}_{\mB_n\setminus \Lambda}}\partial\mB_{n/2}]\asymp {\rm Hm}(x_0,u;\Lambda\cup\{u\}).
\end{equation}
The proof is similar to that of Lemma~\ref{le:gff} by employing GFF (indeed easier as the cluster containing $u$ reaches far away from $\partial\Lambda$ now).
More specifically,
let $\mphi_n^\Lambda$ be the GFF on $\mB_n\setminus\Lambda$ with zero boundary conditions. 
Using the isomorphism theorem again, we have
for all $u\in\Gamma$,
\begin{equation}\label{eq:iso'}
    \Pb[u\connect{\wt{\mathcal L}_{\mB_n\setminus \Lambda}}\partial\mB_{n/2}]=2\Pb[u\connect{\mphi_n^\Lambda>0}\partial\mB_{n/2}].
\end{equation}
When $\mphi_n^\Lambda(u)>0$, let $\mathcal C_u$ be the cluster in $\{z:\mphi_n^\Lambda(z)>0\}$ containing $u$, with the convention that $\mathcal C_u=\emptyset$ if $\mphi_n^\Lambda(u)\le0$, then similar to the proof of Lemma \ref{le:gff}, we have the following inclusion of events for some $0<c_-<c_+<\infty$,
\begin{equation}\label{eq:sandwich'}
\begin{split}
    \{R^{\rm eff}(x_0,\Lambda\cup\{u\})-R^{\rm eff}(x_0,\Lambda\cup\mathcal C_u)>c_+\} \subset
    \{u\connect{\mphi_n^\Lambda>0}\partial\mB_{n/2}\}\\
    \subset\{R^{\rm eff}(x_0,\Lambda\cup\{u\})-R^{\rm eff}(x_0,\Lambda\cup\mathcal C_u)>c_-\}.
\end{split}
\end{equation}
Applying \cite[Corollary 1]{lupu2018random}, we obtain for any constant $c>0$,
$$\begin{aligned}
    \Pb[R^{\rm eff}(x_0,\Lambda\cup\{u\})-R^{\rm eff}(x_0,\Lambda\cup\mathcal C_u)>c\mid\mphi_n^\Lambda(u)]=&\ \Pb\left[W_t\le \mathbbm E[\mphi_n^\Lambda(x_0)\mid\mphi_n^\Lambda(u)],\ \forall 0<t<c\right]\\
    \asymp&\ {\rm Hm}(x_0,u;\Lambda\cup\{u\}) (\mphi_n^\Lambda(u)\vee0).
\end{aligned}$$
Taking the expectation on both sides, and noting that ${\rm Var}(\mphi_n^\Lambda(u))\asymp 1$ (since $1\le \dist(u,\Lambda)\le 2$ by definition), we have
$$\Pb[R^{\rm eff}(x_0,\Lambda\cup\{u\})-R^{\rm eff}(x_0,\Lambda\cup\mathcal C_u)>c]\asymp{\rm Hm}(x_0,u;\Lambda\cup\{u\}).$$
This, combined with \eqref{eq:iso'} and (\ref{eq:sandwich'}), 
gives \eqref{eq:two-arm'}.

Next, we use \eqref{eq:two-arm'} to get \eqref{eq:Nc}.    
On the one hand, ${\rm Hm}(x_0,u;\Lambda\cup\{u\})\ge{\rm Hm}(x_0,u;\Gamma)$, then by (\ref{eq:two-arm'}), for some $c>0$,
    $$\Nc(\Lambda)\gtrsim\sum_{u\in\Gamma}{\rm Hm}(x_0,u;\Gamma)={\rm Hm}(x_0,\Gamma)\ge{\rm Hm}(x_0,\mB_{n/8})\ge c.$$
    On the other hand, for all $u\in \Gamma$, there exists $v_u\in\partial\Lambda$ such that ${\rm dist}(u,v_u)\le2$ (arbitrarily choose one if there are multiple options). Thus, for all $u\in\Gamma$,
    $${\rm Hm}(x_0,v_u;\Lambda)\ge{\rm Hm}(x_0,u;\Lambda\cup\{u\})\cdot{\rm Hm}(u,v_u;\Lambda)\ge4^{-4}{\rm Hm}(x_0,u;\Lambda\cup\{u\}).$$
    Summing over $u\in\Gamma$ on both sides of the above inequality, we get
    $$\Nc(\Lambda)\le 4^4\sum_{u\in\Gamma}{\rm Hm}(x_0,v_u;\Lambda)\le 4^4\cdot 25\,{\rm Hm}(x_0,\Lambda)\le 4^4\cdot25,$$
    where we used the fact that $|\{u\in\Gamma:v_u=v\}|\le25$ for all $v\in\Lambda$ and ${\rm Hm}(x_0,\Lambda)\le 1$.
\end{proof}

With Lemma~\ref{lem:NG} at hand, we can now turn to Proposition~\ref{prop:k=1} itself. 

\begin{proof}[Proof of Proposition~\ref{prop:k=1}]
We begin with the lower bound.
For all $z\in\widetilde A_{n/8,3n/8}\cap\mathbb Z^d$, define a four-arm event around $z$ by
$$\widetilde{\mathcal A}_z:=\{\mbox{there are at least two outermost clusters in }\widetilde{\mathcal L}_n\mbox{ crossing }\widetilde A_{2,n/8}(z)\}.$$
Then, by the locality of four-arm events \cite[Proposition~5.3]{GNQ1}, 
    \begin{equation}\label{eq:locality}
        \Pb[\widetilde{\mathcal A}_z]\lesssim \widetilde\pi_4(2,n/8).
    \end{equation}
    Define $N=\sum_{z\in\widetilde A_{n/8,3n/8}\cap\mathbb Z^2}\mathbbm1_{\widetilde{\mathcal A}_z}$. 
    It follows from \eqref{eq:locality} and Lemma~\ref{lem:qua1} that
    \begin{equation}\label{eq:expNupper}
        \mathbbm E[N]\lesssim n^2\widetilde\pi_4(2,n/8)\lesssim n^2\widetilde\pi_4(n).
    \end{equation}

\begin{figure}[h]
\centering
\includegraphics[width=0.58\linewidth]{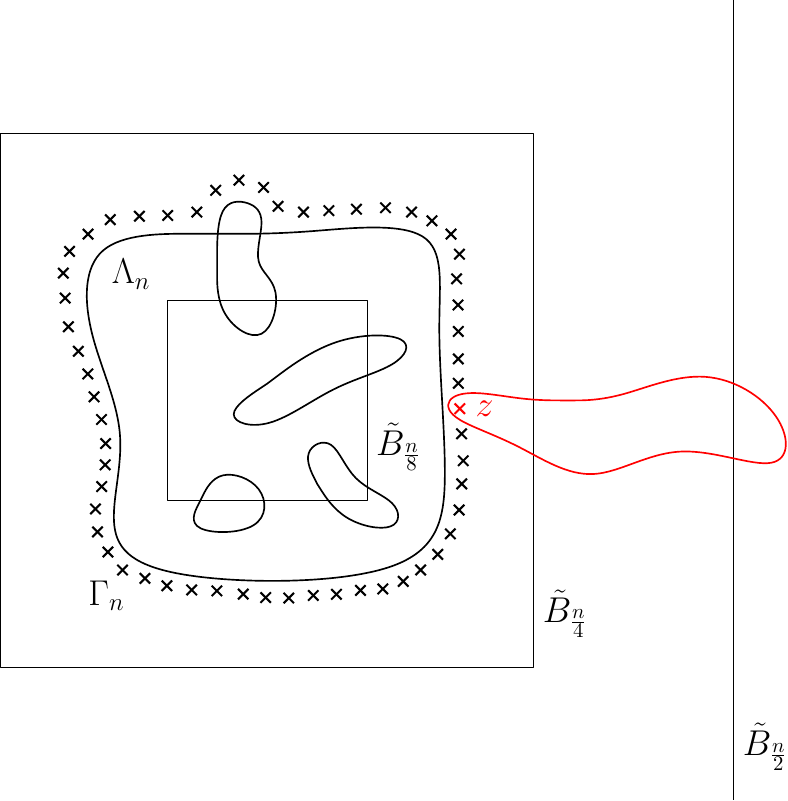}
\caption{\label{fig:lower_bound}
Proof of Proposition \ref{prop:k=1}. Our strategy is to ensure first that $\wt F_n$ occurs, and then find all $z \in \Gamma_n$ such that $z \connect{\wt{\mathcal L}_{\mB_n\setminus\Lambda_n}} \partial\mB_{n/2}$.}
\end{figure}

    Next, we will bound $\mathbbm E[N]$ from below by Lemma~\ref{lem:NG}. For this, we define $\Lambda_n$ as the union of $\mB_{n/8}$ and all the clusters in $\wt{\mathcal L}_{n}$ that intersect $\mB_{n/8}$, and let $\Gamma_n=\partial_\out\{z\in\mathbb Z^2:{\rm dist}(z,\Lambda_n)\le 2\}$ (see Figure~\ref{fig:lower_bound} for an illustration). 
    Recall $\wt F_n$ from Lemma \ref{lem:cluster}, which is the event that there exists a cluster of $\wt\Lc_{n}$ that intersects $\wt B_{n/8}$ with outer boundary contained in $\wt A_{n/8,n/4}$.
    Then, $\widetilde F_n$ is measurable with respect to $\Lambda_n$, and occurs with uniformly positive probability by Lemma \ref{lem:cluster}. 
    Moreover, on $\widetilde F_n$, for any vertex $z\in\Gamma_n$, we have $\{z\connect{\wt{\mathcal L}_{\mB_n\setminus\Lambda_n}}\partial\mB_{n/2}\}\subset\widetilde{\mathcal A}_z$. 
    Hence, applying Lemma \ref{lem:NG}, we obtain that on $\widetilde F_n$,
    $$\mathbb E[N\mid\Lambda_n]\ge\mathbb E[\Nc(\Lambda_n)\mid\Lambda_n]>c.$$ 
    Taking expectations above and using Lemma \ref{lem:cluster}, we get
    $\mathbbm E[N]>c'$ for some $c'>0$. This, together with (\ref{eq:expNupper}), yields the lower bound $\widetilde\pi_4(n)\gtrsim n^{-2}$.

    Now we prove the upper bound. In this case, we consider the event $\wt\Ac_z'$ that there are two outermost clusters $\Cc_1$ and $\Cc_2$ in $\wt\Lc_n$ such that $\Cc_1$ connects $\mB_1(z)$ and $\mB_{n/8}$ with its outer boundary contained in $\wt
    A_{n/8,n/4}$, and $\Cc_2$ connects $\mB_1(z)$ and $\partial\mB_{n/2}$.
    We claim that for all $z\in\widetilde A_{n/7,n/5}$ (we will turn to it later),
    \begin{equation}\label{eq:A'}
        \Pb[\widetilde{\mathcal A}_z']\gtrsim \widetilde\pi_4(n).
    \end{equation}  
    Letting $N'=\sum_{z\in\widetilde A_{n/8,n/4}\cap\mathbb Z^2}\mathbbm1_{\widetilde{\mathcal A}_z'}$, we have by \eqref{eq:A'} that
    \begin{equation}\label{eq:N'lower}
        \mathbbm E[N']\gtrsim n^2\widetilde\pi_4(n).
    \end{equation}
    On the other hand, noting that $\{N'\ge1\}\subset\widetilde F_n$, we have
    \begin{equation}\label{eq:N'upper}
        \mathbbm E[N']\le\mathbbm E[N'\mid\widetilde F_n]\le\mathbbm E\left[9\mathbbm E[\Nc(\Lambda_n)\mid\Lambda_n]\big|\widetilde F_n\right]<C.
    \end{equation}
    Here, the last inequality follows from Lemma \ref{lem:NG}, and the second one can be derived as follows. On the event $\widetilde F_n$, for each $z\in A_{n/8,n/4}$ such that $\widetilde{\mathcal A}_z'$ happens, we have $\mB_1(z)\subset\{w\in\wt\Z^2:\dist(w,\Lambda_n)\le2\}$ since
    the cluster $\mathcal C_1$ in the definition of $\widetilde{\mathcal A}_z'$ must intersect $\mB_1(z)$ and be contained in $\Lambda_n$; 
    meanwhile, there is a distinct cluster $\mathcal C_2$ connecting $\mB_1(z)$ and $\partial\mB_{n/2}$ which implies the existence of some $w_z\in B_1(z)\cap\Gamma_n$ such that $w_z\connect{\wt{\mathcal L}_{\mB_n\setminus\Lambda_n}}\partial\mB_{n/2}$. The second inequality then follows since $|\{z:w_z=w\}|\le9$ for all $z\in A_{n/8,n/4}$ satisfying $\widetilde{\mathcal A}_z'$. Now, combining (\ref{eq:N'lower}) and (\ref{eq:N'upper}), we obtain the upper bound $\widetilde\pi_4(n)\lesssim n^{-2}$ and conclude the proof.
\end{proof}



\begin{figure}[h]
\centering
\includegraphics[width=0.8\linewidth]{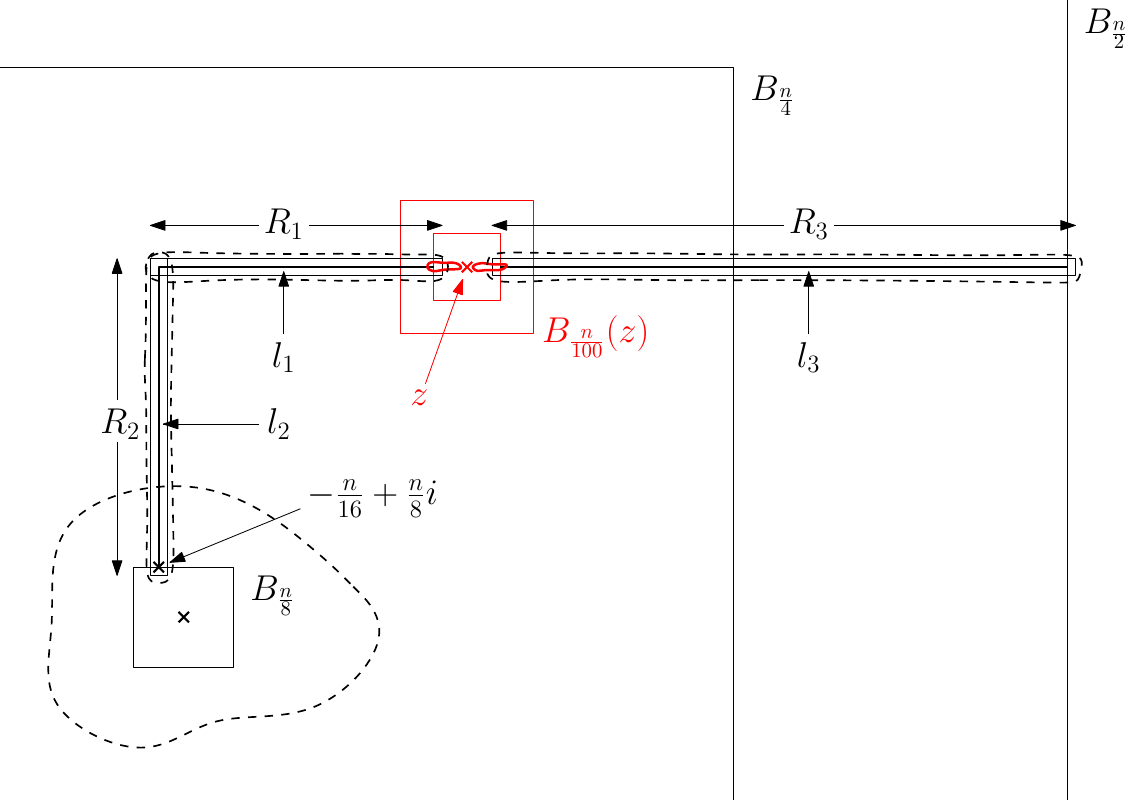}
\caption{\label{fig:gluing}This figure shows the explicit geometric construction used in the proof of (\ref{eq:A'}), based on separation and gluing arguments.}
\end{figure}

Finally, we now complete the proof of Proposition \ref{prop:k=1} by establishing the remaining claim. 

\begin{proof}[Proof of \eqref{eq:A'}]
    We prove it by concrete constructions (separation and gluing).
    Without loss of generality, assume that $\Re(z),\Im(z)\ge0$.
    Let $L[x,y]$ be the line segment connecting the points $x$ and $y$ in $\R^2$.
    Then, we consider three line segments: $l_1:=L[z-n/200,-n/16+\Im(z)i]$, $l_2:=L[-n/16+\Im(z)i,-n/16+ni/8]$, and $l_3:=L[z+n/200,n/2+\Im(z)i]$. For $j=1,2,3$, we define the boxes $R_j:=\mB_{n/4000}(l_j)$ and $R_j^*:=\mB_{n/2000}(l_j)$, where $\mB_r(A):=\bigcup_{z\in A}\mB_r(z)$. 
    Define the event 
    \[
    \wt H_j:=\{ \text{there exists a loop in $\wt{\mathcal L}_{R_j^*}$ that surrounds $R_j$} \}.
    \]
    Let $\widetilde{\mathcal A}_z''$ be the translation of the event $\widetilde{\mathcal A}_{\rm loc}^{\rm sep}(1,n/200)$ from the center $0$ to $z$, which is thus measurable with respect to  $\wt\Lc_{\mB_{n/100}(z)}$. 
    Define the metric-graph version of the event $E_1(n)$ from the proof of Lemma \ref{lem:cluster} by (hence also occurs with uniformly positive probability) 
    \[
    \wt E_1(n):=\{ \text{there exists a loop of $\wt\Lc_{n/7}$ that surrounds $\mB_{n/8}$} \}.
    \]
    Furthermore, let $G$ be the event
    $$G:=\{\mbox{all clusters in }\wt\Lc_n\setminus(\wt\Lc_{n/7}\cup\bigcup_{j=1}^3\wt\Lc_{R_j^*}\cup\wt\Lc_{\mB_{n/100}(z)}\mbox{ have diameter smaller than }n/4000\}.$$
    Then, we observe that $\widetilde{\mathcal A}_z'\supset\widetilde{\mathcal A}_z''\cap\bigcap_{j=1}^3\widetilde H_{\widetilde B_j}\cap E_1(n)\cap G$, and that the events involved in the intersection are mutually independent, hence it suffices to bound their probabilities from below. 
    
    By Lemmas~\ref{lem:sep} and~\ref{lem:qua1} (together with translation invariance), we have $\Pb[\widetilde{\mathcal A}_z'']\gtrsim \widetilde\pi_4(n/200)\gtrsim \widetilde\pi_4(n)$.
    Similar to Lemma \ref{lem:loop}, we have
    $\Pb[\widetilde H_j]\ge c$
    for all $1\le j\le3$, 
    and by \cite[Lemma~3.2]{GNQ1} again, we have $\Pb[G]>c'$. 
    Here $c,c'$ are universal constants. Hence,
    $$\begin{aligned}
        \Pb[\widetilde{\mathcal A}_z']\ge\Pb\left[\widetilde{\mathcal A}_z''\cap\bigcap_{j=1}^3\widetilde H_j\cap \wt E_1(n)\cap G\right]=\Pb[\widetilde{\mathcal A}_z'']\cdot\prod_{j=1}^3\Pb[\widetilde H_j]\cdot\Pb[\wt E_1(n)]\cdot\Pb[G]\gtrsim \widetilde\pi_4(n),
    \end{aligned}$$
    finishing the proof.
\end{proof}

\subsection{Proof of \eqref{eq:d-four-arm}}\label{subsec:d_four_arm}
We now carry out the proof of \eqref{eq:d-four-arm} by using ideas similar to those presented in Section~\ref{subsec:m_four_arm}.
The key ingredient is the following generalization of Lemma \ref{lem:NG} in the discrete setup, where for an integer $k$, we use the notation $\partial B_k:=\{v \in \Z^2:\|v\|_{\infty} = k\}$.
\begin{lemma}\label{le:d-NG}
    Let $1\le k\le n/8$.
    Let $\Lambda$ be a subset of $\mathbb Z^2$ such that $B_{n/8}\subset\Lambda\subset B_{n/4}$. Take $x_0=\lfloor \frac{3n}{4}\rfloor i$, and define $$\Gamma:=\{u\in B_n:0< {\rm dist}(u,\Lambda)\le k, {\rm Hm}(x_0,u;\Lambda\cup\{u\})>0\},$$
    where ${\rm Hm}(\cdot)={\rm Hm}^n(\cdot )$ is defined in an analogous way to \eqref{eq:def_hm}, using a simple random walk on $\Z^2$, instead of a Brownian motion on $\wt\Z^2$. Let $\Nc^k(\Lambda)$ be the quantity
    $$\Nc^k(\Lambda):=\sum_{u\in\Gamma}\Pb[B_k(u)\connect{\mathcal L_{B_n\setminus\Lambda}}\partial B_{n/2}].$$
    Then, there exists a universal constant $c>0$ such that $\Nc^k(\Lambda)<ck^2$.
\end{lemma}
\begin{proof}
Let $\wt\Lambda$ be the metric graph associated with $\Lambda$. 
For any $u\in\Gamma$, let $\wt\Ec_u$ be the set of boundary-to-boundary excursions on $\mB_n\setminus\wt\Lambda$ with intensity $\frac12 \sum_{x,y \in \partial\wt\Lambda\cap \mB_{2k}(u)} \wt\nu_{x,y}^{\mB_n\setminus\wt\Lambda}$, independent of the loop soups.
Since the RWLS is stochatically dominated by the BLS on the associated metric graph, we have 
\begin{equation*}
    \Pb[B_k(u)\connect{\mathcal L_{B_n\setminus\Lambda}}\partial B_{n/2}]
    \le \Pb[B_k(u)\connect{\wt\Lc_{\mB_n\setminus\wt\Lambda}}\partial B_{n/2}].
\end{equation*}
Let $E_u$ be the event that there exists an excursion in $\wt\Ec_u$ disconnecting $B_k(u)\setminus\Lambda$ from $\partial B_{n/2}$ inside $\mB_n\setminus\wt\Lambda$, which has probability uniformly bounded away from zero. Therefore, similar to (\ref{eq:exc-disc}), we obtain that
\begin{equation*}
\Pb[B_k(u)\connect{\wt\Lc_{\mB_n\setminus\wt\Lambda}}\partial B_{n/2}]
    \lesssim\Pb[\{B_k(u)\connect{\wt{\mathcal L}_{\mB_n\setminus\wt\Lambda}}\partial B_{n/2}\}\cap E_u]
    \le \Pb[\partial\wt\Lambda\cap \wt B_{2k}(u)\connect{\wt{\mathcal L}_{\mB_n\setminus\wt\Lambda}\oplus\wt\Ec_u}\partial B_{n/2}].
\end{equation*}
Now, let $\mphi_u$ denote the GFF on $\wt B_n\setminus\wt \Lambda$ with boundary conditions
$$\mphi_u(z)=\begin{cases}
1,&z\in\partial\wt\Lambda\cap\mB_{2k}(u);\\
0,&(\partial\wt\Lambda\setminus \wt B_{2k}(u))\cup\partial\wt B_n.\end{cases}$$
Analogous to \eqref{eq:iso}, we have 
\begin{equation*}
    \Pb[\partial\wt\Lambda\cap \wt B_{2k}(u)\connect{\mathcal L_{\mB_n\setminus\wt\Lambda}\oplus\wt\Ec_u}\partial B_{n/2}]
    = \Pb[\partial\wt\Lambda\cap \wt B_{2k}(u)\connect{\mphi_u>0}\partial B_{n/2}].
\end{equation*}
Similar to \eqref{eq:sandwich'}, there exists a constant $c>0$ such that
    \begin{equation*}
        \{\partial\wt\Lambda\cap \wt B_{2k}(u)\connect{\mphi_u>0}\partial B_{n/2}\}\subset\{R^{\rm eff}(x_0,\partial\wt\Lambda\cup \partial \mB_{n/2})-R^{\rm eff}(x_0,\partial\wt\Lambda\cup\widetilde\Lambda_0^*\cup \partial \mB_{n/2})>c\},
    \end{equation*}
where $\wt\Lambda_0^*$ is the union of clusters in $\{z\in\mB_n\setminus\wt\Lambda:\mphi_u(z)>0\}$ that intersect $\partial\wt\Lambda\cap\mB_{2k}(u)$.
Using \cite[Corollary 1]{lupu2018random} again, we conclude
\begin{equation*}
    \Pb[\partial\wt\Lambda\cap \wt B_{2k}(u)\connect{\mphi_u>0}\partial B_{n/2}]\lesssim \mathbb E[\mphi_u(x_0)]={\rm Hm}(x_0,\partial\wt\Lambda\cap \wt B_{2k}(u);\Lambda),
\end{equation*}
Note that since $\Lambda\subseteq \Z^2$, $\partial \wt \Lambda \cap \wt B_{2k}(u)$ is also a subset of $\Z^2$.
Combining all the above inequalities, we get the upper bound
$$\Pb[B_k(u)\connect{\mathcal L_{B_n\setminus\Lambda}}\partial B_{n/2}]\lesssim{\rm Hm}(x_0,\partial\wt\Lambda\cap \wt B_{2k}(u);\Lambda).$$
Summing over $u\in\Gamma$ on both sides above yields
$$\Nc^k(\Lambda)\lesssim \sum_{u\in\Gamma}{\rm Hm}(x_0,\partial\wt\Lambda\cap \wt B_{2k}(u);\Lambda)=\sum_{z\in\partial\wt\Lambda}{\rm Hm}(x_0,z;\Lambda)\cdot|\Gamma\cap B_{2k}(z)|\lesssim k^2\,{\rm Hm}(x_0,\partial \wt\Lambda;\Lambda)\le k^2,$$
proving the lemma.
\end{proof}

Finally, we finish the proof of (\ref{eq:d-four-arm}), which is analogous to the proof of the upper bound in (\ref{eq:four-arm1}).
\begin{proof}[Proof of (\ref{eq:d-four-arm})]
    By Lemma~\ref{lem:qua1}, we can assume that $k<10^{-10}n$. For all $z\in A_{n/7,n/5}$, consider the event $\mathcal A_z(k)$ that there are two outermost clusters $\mathcal C_1$ and $\mathcal C_2$ in $\mathcal L_{n}$ such that $\mathcal C_1$ connects $B_k(z)$ and $B_{n/8}$ with its outer boundary contained in $A_{n/8,n/4}$, and $\mathcal C_2$ connects $B_k(z)$ and $\partial B_{n/2}$. Then, the following counterpart of (\ref{eq:A'}) holds
    \begin{equation}\label{eq:zk4}
        \Pb[\mathcal A_z(k)]\gtrsim\pi_4(k,n).
    \end{equation} 
    On the one hand, if we define 
    $N^k:=\sum_{z\in A_{n/7,n/5}}\mathbbm 1_{\mathcal A_z(k)}$, then it satisfies $\mathbb E[N^k]\gtrsim n^2\pi_4(k,n)$ by (\ref{eq:zk4}).
    On the other hand, if we let $\Lambda_n'$ be the union of $B_{n/8}$ and all the clusters in $\Lc_n$ that intersect $B_{n/8}$, then $\mathbb E[\Nc^k(\Lambda_n')\mathbbm1_{\partial_{\out}\wt\Lambda_n'\subset A_{n/8,n/4}}]\lesssim k^2$ by Lemma \ref{le:d-NG}. When $N^k>0$, we have $\partial_{\out}\wt\Lambda_n'\subset A_{n/8,n/4}$ and $N^k=\Nc^k(\Lambda_n')$, and therefore 
    $$n^2\pi_4(k,n)\lesssim \mathbb E[N^k]\le \mathbb E[\Nc^k(\Lambda_n')\mathbbm1_{\partial_{\out}\wt\Lambda_n'\subset A_{n/8,n/4}}]\lesssim k^2.$$
    This completes the proof.
\end{proof}

\subsection*{Acknowledgments}
We are grateful to Tiancheng He for helpful discussions.
YB is partially supported by the elite undergraduate training program of School of Mathematical Sciences at Peking University. YG and WQ are supported by National Key R\&D Program of China (No.\ 2023YFA1010700) and a grant from City University of Hong Kong (Project No.\ 7200745). PN is partially supported by a GRF grant from the Research Grants Council of the Hong Kong SAR (project CityU11312325). WQ is also partially supported by a GRF grant from the Research Grants Council of the Hong Kong SAR (project CityU11308624).
	
\bibliographystyle{plain}
\bibliography{references}

\begin{thebibliography}{10}

\bibitem{aru2020first}
Juhan Aru, Titus Lupu, and Avelio Sep{\'u}lveda.
\newblock The first passage sets of the 2d gaussian free field: convergence and
  isomorphisms.
\newblock {\em Communications in Mathematical Physics}, 375(3):1885--1929,
  2020.

\bibitem{BFS82Loop}
David Brydges, Jürg Fröhlich, and Tom Spencer.
\newblock The random walk representation of classical spin systems and
  correlation inequalities.
\newblock {\em Comm. Math. Phys.}, 83(1):123--150, 1982.

\bibitem{chelkak2016robust}
Dmitry Chelkak.
\newblock {Robust discrete complex analysis: A toolbox}.
\newblock {\em The Annals of Probability}, 44(1):628 -- 683, 2016.

\bibitem{ding2022crossing}
Jian Ding, Mateo Wirth, and Hao Wu.
\newblock Crossing estimates from metric graph and discrete gff.
\newblock In {\em Annales de l'Institut Henri Poincare (B) Probabilites et
  statistiques}, volume~58, pages 1740--1774. Institut Henri Poincar{\'e},
  2022.

\bibitem{duminil2022planar}
Hugo Duminil-Copin and Ioan Manolescu.
\newblock Planar random-cluster model: scaling relations.
\newblock In {\em Forum of Mathematics, Pi}, volume~10, page e23. Cambridge
  University Press, 2022.

\bibitem{Dynkin1984Isomorphism}
Evgeniy Dynkin.
\newblock Gaussian and non-{G}aussian random fields associated with {M}arkov
  processes.
\newblock {\em J. Funct. Anal.}, 55:344--376, 1984.

\bibitem{Dynkin1984IsomorphismPresentation}
Evgeniy Dynkin.
\newblock Local times and quantum fields.
\newblock In {\em Seminar on Stochastic Processes, Gainesville 1983}, volume~7
  of {\em Progress in Probability and Statistics}, pages 69--84. Birkhauser,
  1984.

\bibitem{EKMRS2000}
Nathalie Eisenbaum, Haya Kaspi, Michael~B. Marcus, Jay Rosen, and Zhan Shi.
\newblock A {R}ay-{K}night theorem for symmetric {M}arkov processes.
\newblock {\em Ann. Probab.}, 28(4):1781--1796, 2000.

\bibitem{GNQ1}
Yifan Gao, Pierre Nolin, and Wei Qian.
\newblock {Percolation of discrete GFF in dimension two I. Arm events in the
  random walk loop soup}.
\newblock {\em arXiv preprint arXiv:2409.16230}, 2024.

\bibitem{GNQ2}
Yifan Gao, Pierre Nolin, and Wei Qian.
\newblock {Percolation of discrete GFF in dimension two II. Connectivity
  properties of two-sided level sets}.
\newblock {\em arXiv preprint arXiv:2409.16273}, 2024.

\bibitem{GNQ3}
Yifan Gao, Pierre Nolin, and Wei Qian.
\newblock Up-to-constants estimates on four-arm events for simple conformal
  loop ensemble.
\newblock {\em arXiv preprint arXiv:2504.06202}, 2025.

\bibitem{lawler2007random}
Gregory Lawler and Jos{\'e} Trujillo~Ferreras.
\newblock Random walk loop soup.
\newblock {\em Transactions of the American Mathematical Society},
  359(2):767--787, 2007.

\bibitem{lawler2011defining}
Gregory~F Lawler.
\newblock Defining {SLE} in multiply connected domains with the {B}rownian loop
  measure.
\newblock {\em arXiv preprint arXiv:1108.4364}, 2011.

\bibitem{lawler2010random}
Gregory~F Lawler and Vlada Limic.
\newblock {\em Random walk: a modern introduction}, volume 123.
\newblock Cambridge University Press, 2010.

\bibitem{MR2815763}
Yves Le~Jan.
\newblock {\em Markov paths, loops and fields}, volume 2026 of {\em Lecture
  Notes in Mathematics}.
\newblock Springer, Heidelberg, 2011.
\newblock Lectures from the 38th Probability Summer School held in Saint-Flour,
  2008, {\'E}cole d'{\'E}t{\'e} de Probabilit{\'e}s de Saint-Flour.
  [Saint-Flour Probability Summer School].

\bibitem{lupu2016loop}
Titus Lupu.
\newblock {From loop clusters and random interlacements to the free field}.
\newblock {\em The Annals of Probability}, 44(3):2117 -- 2146, 2016.

\bibitem{lupu2018random}
Titus Lupu and Wendelin Werner.
\newblock The random pseudo-metric on a graph defined via the zero-set of the
  gaussian free field on its metric graph.
\newblock {\em Probability Theory and Related Fields}, 171(3):775--818, 2018.

\bibitem{MarcusRosen2006MarkovGaussianLocTime}
Michael~B. Marcus and Jay Rosen.
\newblock {\em Markov processes, {G}aussian processes and local times}, volume
  100.
\newblock Cambridge University Press, 2006.

\bibitem{SW2012}
Scott Sheffield and Wendelin Werner.
\newblock Conformal loop ensembles: the {M}arkovian characterization and the
  loop-soup construction.
\newblock {\em Ann. of Math. (2)}, 176(3):1827--1917, 2012.

\bibitem{Symanzik66Scalar}
Kurt Symanzik.
\newblock Euclidean quantum field theory {I}. {E}quations for a scalar model.
\newblock {\em J. Math. Phys}, 7(3):510--525, 1966.

\bibitem{Symanzik1969QFT}
Kurt Symanzik.
\newblock Euclidean quantum field theory.
\newblock In {\em Scuola intenazionale di Fisica Enrico Fermi. XLV Corso.},
  pages 152--223. Academic Press, 1969.

\bibitem{MR2892408}
Alain-Sol Sznitman.
\newblock An isomorphism theorem for random interlacements.
\newblock {\em Electron. Commun. Probab.}, 17(9):1--9, 2012.

\bibitem{MR2932978}
Alain-Sol Sznitman.
\newblock {\em Topics in occupation times and {G}aussian free fields}.
\newblock Zurich Lectures in Advanced Mathematics. European Mathematical
  Society (EMS), Z\"{u}rich, 2012.

\bibitem{MR3167123}
Alain-Sol Sznitman.
\newblock On scaling limits and {B}rownian interlacements.
\newblock {\em Bull. Braz. Math. Soc. (N.S.)}, 44(4):555--592, 2013.

\bibitem{Wolpert2}
Robert~L Wolpert.
\newblock Local time and a particle picture for {E}uclidean field theory.
\newblock {\em J. Funct. Anal.}, 30(3):341--357, 1978.

\end{thebibliography}

\end{document}